\documentclass{amsart}

\usepackage{amsfonts, amsmath, amscd}
\usepackage[psamsfonts]{amssymb}
\usepackage{amssymb}
\usepackage{pb-diagram}
\usepackage[usenames]{color}
\usepackage{hyperref}
\usepackage[all]{xy}
\usepackage{graphicx}

\numberwithin{equation}{section}

\newtheorem{theorem}{Theorem}[section]
\newtheorem{corollary}[theorem]{Corollary}
\newtheorem{proposition}[theorem]{Proposition}
\newtheorem{lemma}[theorem]{Lemma}
\newtheorem{remark}[theorem]{Remark}
\theoremstyle{definition}
\newtheorem{example}[theorem]{Example}
\newtheorem{definition}[theorem]{Definition}
\newtheorem{question}[theorem]{Question}

\theoremstyle{plain}

\newcommand{\R}{\mathbb{R}}

\def\N{\mathbb N}
\def\R{\mathbb R}

\def\U{\mathbb U}
\def\Z{\mathbb Z}

\DeclareMathOperator*{\supp}{\rm supp}
\DeclareMathOperator{\Int}{\rm Int}

\def\cont{\mathfrak c}

\def\hull#1{\langle#1\rangle}

\begin{document}

\title[Density Character of Subgroups of Topological Groups]{Density Character of Subgroups of Topological Groups}
\author{Arkady G. Leiderman, Sidney A. Morris, and Mikhail G. Tkachenko}
\address{Department of Mathematics, Ben-Gurion University of the Negev, Beer Sheva, P.O.B. 653, Israel}
\email{arkady@math.bgu.ac.il}
\address{Faculty of Science, Federation University Australia, P.O.B. 663, Ballarat, Victoria, 3353,  Australia
\newline
Department of Mathematics and Statistics, La~Trobe University, Melbourne, Victoria, 3086, Australia}
\email{morris.sidney@gmail.com}
\address{Departamento de Matem\'aticas, Universidad Aut\'onoma Metropolitana, Av. San Rafael Atlixco 186, Col. Vicentina, Del. Iztapalapa, C.P. 09340, M\'exico, D.F., Mexico}
\email{mich@xanum.uam.mx}

\begin{abstract}
It is well-known that a subspace $Y$ of a separable metrizable space $X$ is separable, but without $X$ being assumed metrizable this is not true even in the case that $Y$ is a closed linear subspace of a topological vector space $X$. Early this century K.H.~Hofmann and S.A.~Morris introduced the class of \textit{pro-Lie groups} which  consists of projective limits of finite-dimensional Lie groups and proved that it contains all compact groups, all locally compact abelian groups, and all connected locally compact groups and is closed under the formation of products and closed subgroups. They defined a topological group $G$ to be \textit{almost connected} if the quotient group of $G$ by the connected component of its identity is compact. 

We prove that an almost connected pro-Lie group is separable if and only if its weight is not greater than the cardinality $\cont$ of the continuum. It is deduced from this that an almost connected pro-Lie group is separable if and only if it is homeomorphic to a subspace of a separable Hausdorff space. It is also proved that a locally compact (even feathered) topological group $G$ which is  a subgroup of a separable Hausdorff topological group is separable, but the conclusion is false if it is assumed only that $G$ is homeomorphic to a subspace of a separable Tychonoff space.  It is shown that every precompact (abelian) topological group of weight less than or equal to $\cont$ is topologically isomorphic to a closed subgroup of a separable pseudocompact (abelian) group of weight $\cont$. This result implies that there is a wealth of closed non-separable subgroups of separable pseudocompact groups. An example is also presented under the Continuum Hypothesis of a separable countably compact abelian group which contains a non-separable closed subgroup. It is proved that the following conditions are equivalent for an $\omega$-narrow topological group $G$: (i) $G$ is homeomorphic to a subspace of a separable regular space; (ii) $G$ is a subgroup of a separable topological group;  (iii) $G$  is a closed subgroup of a separable path-connected locally path-connected topological group.
\end{abstract}

\maketitle
%%%%%%%%%%%%%%%%%%%%%%%%%%%%
\section{Introduction}\label{intro}
All topological spaces and topological groups are assumed to be Hausdorff.
The \textit{weight} $w(X)$ of a topological space $X$ is defined as the smallest cardinal number of the form $|\mathcal{B}|$, where $\mathcal{B}$ is a base of topology in $X$. The \textit{density character} $d(X)$ of a topological space $X$ is $\min\{|A|: A\,\,\mbox{is dense in} \,X \}$.
If $d(X) \leq \aleph_0$, then we say that the space $X$ is \textit{separable}.
 
It is well-known that a subspace $S$ of a separable metrizable space $X$ is separable, but a closed subspace $S$ of a separable Hausdorff topological space $X$ is not necessarily separable. Even a closed linear subspace $S$ of a separable Hausdorff topological vector space $X$ can fail to be separable \cite{Lohman}. 

In several classes of topological groups, the situation improves notably.
It has been proved by W. Comfort and G. Itzkowitz \cite{Comfort_Itzkowitz} that a closed subgroup $S$ of a separable locally compact topological group $G$ is separable.
It is known also that a metrizable subgroup of a separable topological group is separable \cite[p.~89]{Vidossich} (see also \cite{Lohman}). 

In Section~\ref{Positive} of this article we look at conditions on the topological group $G$ which are sufficient to guarantee its separability if $G$ is a subspace of a separable Hausdorff or regular space $X$. En route we show in Theorem~\ref{Theorem 10} that an almost connected pro-Lie group (for example, a connected locally compact group or a compact group) is separable if and only if its weight is not greater than $\cont$. One cannot drop \lq{almost connected\rq} in the latter result since by Corollary~\ref{Cor:subg}, there exists a separable \textit{prodiscrete} abelian (hence pro-Lie) group which contains a closed non-separable subgroup. Let us note that a prodiscrete group is a complete group with a base at the identity consisting of open subgroups. Clearly closed subgroups of prodiscrete groups are prodiscrete. It is proved in Theorem~\ref{Th:Main} that if an almost connected pro-Lie group $G$ is a subspace of a separable space, then $G$ is also separable. 

In the rest of Section~\ref{Positive} we consider subgroups of separable topological groups. It is shown in Theorem~\ref{Th:Fea} that every feathered subgroup of a separable group is separable. The class of feathered groups contains both locally compact and metrizable groups, thus we obtain a generalization of aforementioned results of \cite{Comfort_Itzkowitz} and \cite{Vidossich}, \cite{Lohman}. 

In Section~\ref{Sec2} we consider closed isomorphic embeddings into separable groups.
We prove in Theorem~\ref{Th:nA} that a precompact topological group of weight $\leq\cont$ is topologically isomorphic to a closed subgroup of a separable pseudocompact group $H$ of weight $\leq\cont$. Since there is a wealth of non-separable pseudocompact groups of weight $\cont$, we conclude that closed subgroups of separable pseudocompact groups can fail to be separable. The proof of Theorem~\ref{Th:nA} is somewhat technical, so we find it convenient to supply the reader with a considerably easier proof of this theorem for the abelian case (see Theorem~\ref{Th:ab}).

We also present in Proposition~\ref{Pr:CoCo}, under the Continuum Hypothesis, an example of a separable countably compact abelian group $G$ which contains a non-separable closed subgroup. We do not know if such an example exists in $ZFC$ alone (see Question~\ref{Example 3}).

A topological group which has a local base at the identity element consisting of open subgroups is called \textit{protodiscrete}. A complete protodiscrete group is said to be \textit{prodiscrete}. It is shown in Section~\ref{Sec:Pd} that closed subgroups of separable prodiscrete abelian groups can fail to be separable (see Corollary~\ref{Cor:subg}). Since prodiscrete abelian groups are pro-Lie, we see that closed subgroups of separable pro-Lie groups need not be separable either. 

In Section~\ref{Sec:Emb} we compare embeddings of topological groups in separable regular spaces and in separable topological groups. It is proved in Theorem~\ref{Th:Emb} that in the class of $\omega$-narrow topological groups, the difference between the two types of embeddings disappears, even if we require an embedding to be closed. In fact, we show that if an $\omega$-narrow topological group $G$ is \textit{homeomorphic} to a subspace of a separable regular space, then $G$ is topologically isomorphic to a \textit{closed} subgroup of a separable path-connected, locally path-connected topological group.

%%%%%%%%%%%%%%%%%%%%%%%%%%%%%%%%
\section{Background Results}\label{Background}
In this section we collect several important well-known facts that will be frequently 
applied in the sequel. 

\begin{theorem}\label{Th:1}{\rm (De Groot, \cite[Theorem~3.3]{Hodel})}
If $X$ is a separable regular space, then $w(X)\le \cont$. More generally, 
every regular space $X$ satisfies $w(X)\le 2^{d(X)}$. 
\end{theorem}

Compact dyadic spaces are defined to be continuous images of generalized Cantor cubes $\{0,1\}^\kappa$, where $\kappa$ is an arbitrary cardinal number. 

\begin{proposition}\label{Pro:Eng}{\rm (Engelking, \cite[Theorem~10]{Eng2})}
Let $\kappa$ be an infinite cardinal. A compact dyadic space $K$ with $w(K)\leq 2^\kappa$ satisfies $d(K)\leq\kappa$. In particular, if $w(K)\leq\cont$, then $K$ is separable.
\end{proposition}

\begin{theorem}\label{Theorem 3}{\rm (Comfort, \cite[Theorem~3.1]{Comfort})} 
Every infinite compact topological group $G$ satisfies $|G|= 2^{w(G)}$ and $d(G)=\log w(G)$. In particular, if a compact group $G$ satisfies $w(G)\leq\cont$ then it is separable.
\end{theorem}

Theorem~\ref{Theorem 3} is deduced in \cite{Comfort} from the fact that compact topological groups are dyadic. In fact as Comfort observes, if $G$ is a compact group of weight $\alpha\ge \aleph_0$, then there are continuous surjections as indicated below:
$$ 
\{0,1\}^\alpha \to G \to [0,1]^\alpha
$$
and $|\{0,1\}^\alpha| = | [0,1]^\alpha| =2^\alpha$, and so $|G|= 2^\alpha$.

\medskip
A  key result we shall need is the following one:

\begin{theorem}\label{HMP}{\rm (Hewitt--Marczewski--Pondiczery, \cite[Theorem~11.2]{Hodel})}
Let $\{X_i: i\in I\}$ be a family of topological spaces and $X=\prod_{i\in I}X_i$, where $|I|\le 2^\kappa$ for some cardinal number $\kappa\geq\omega$. If $d(X_i)\le \kappa$ for each $i\in I$, then $d(X)\le \kappa$. In particular, the product of no more than $\cont$ separable spaces is separable.
\end{theorem}

In the following two definitions we introduce the main concepts of our study in Section~\ref{Positive}.

\begin{definition}\label{almost connected}{\rm (Hofmann--Morris, \cite{PROBOOK})}
{\rm A Hausdorff topological group $G$ is said to be \emph{almost connected} if the quotient group $G/G_0$ is compact, where $G_0$ is the connected component of the identity in $G$.}  
\end{definition}

Clearly the class of almost connected topological groups includes all compact groups, all connected topological groups, and their products.

\begin{definition}\label{pro-Lie group}{\rm (Hofmann--Morris, \cite{PROBOOK})}
{\rm A topological group is called a \emph{pro-Lie} group if it is a projective limit of finite-dimensional Lie groups.}  
\end{definition}

As shown in \cite{PROBOOK} the class of  pro-Lie groups includes all locally compact abelian topological groups, all compact groups, all connected locally compact topological groups, and all almost connected locally compact topological groups. Further, every closed subgroup of a pro-Lie group is again a pro-Lie group and any finite or infinite product of pro-Lie groups is a pro-Lie group.

\begin{theorem}\label{Theorem 7}{\rm (Hofmann--Morris, \cite[Theorem~12.81]{PROBOOK})}
Let $G$ be a connected pro-Lie group. Then $G$ contains a maximal compact connected subgroup $C$ such that $G$ is homeomorphic to $C\times {\R}^\kappa$, for some cardinal $\kappa$. 
\end{theorem}

Almost connected pro-Lie groups also admit a topological characterization similar to the one in Theorem~\ref{Theorem 7}.

\begin{theorem}\label{Th_HM}{\rm (Hofmann--Morris, \cite[Corollary~8.9]{HM2})}
Every almost connected pro-Lie group $G$ is homeomorphic to the product $\R^\kappa \times \{0,1\}^\lambda \times B$, where $\{0,1\}$ is the discrete two-element group, $B$ is a compact connected group, and $\kappa,\lambda$ are cardinals. 
\end{theorem}

\begin{remark}\label{Remark 8} 
{\rm It is known that a topological group $G$ is separable if it contains a closed subgroup $H$ such that both $H$ and the quotient space $G/H$ are separable. In particular, separability is a three space property \cite[Theorem~1.5.23]{AT}.}
\end{remark}

A topological group is said to be \textit{$\omega$-narrow} \cite[Section~3.4]{AT} if it can be covered by countably many translations of every neighborhood of the identity element. It is known that every separable topological group is $\omega$-narrow (see \cite[Corollary~3.4.8]{AT}). The class of $\omega$-narrow groups is productive and hereditary with respect to taking arbitrary subgroups \cite[Section~3.4]{AT}, so $\omega$-narrow groups need not be separable. In fact, $\omega$-narrow groups can have uncountable cellularity (see \cite{Usp} or \cite[Example~5.4.13]{AT}). 

The following theorem characterizes the class of $\omega$-narrow topological groups.

\begin{theorem}\label{Th:Gur}{\rm (Guran, \cite{Gur})}
A topological group $G$ is $\omega$-narrow if and only if $G$ is topologically isomorphic to a subgroup of a product of second countable topological groups.
\end{theorem}

%%%%%%%%%%%%%%%%%%%%%%%%%%%%%%%%%
\section{Separability of pro-Lie groups}\label{Positive}
We prove in this section that an almost connected pro-Lie group $G$ is separable 
if and only if $w(G)\leq \cont$. This fact is then used to show that if an almost connected pro-Lie group $G$ is a subspace of a separable space, then $G$ is separable as well. We also prove in Corollary~\ref{Theorem 14} that a locally compact subgroup of a separable topological group is separable. 

\begin{theorem}\label{Theorem 10} 
An almost connected pro-Lie group $G$ is separable if and only if $w(G) \le \cont$. In particular this is the case if $G$ is a connected locally compact group. 
\end{theorem}

\begin{proof}
If $G$ is separable, then $w(G)\leq\cont$ by Theorem~\ref{Th:1}, since Hausdorff topological groups are regular. 
 
Conversely, assume that $w(G)\leq\cont$. By Theorem~\ref{Th_HM}, the group $G$ is homeomorphic to the product $\R^\kappa\times\{0,1\}^\lambda\times B$, where $B$ is a compact connected topological group and $\kappa,\lambda$ are cardinals. It is clear that $w(G)=\kappa\cdot\lambda\cdot\mu\leq\cont$, where $\mu=w(B)$. Hence the spaces $\R^\kappa$ and $\{0,1\}^\lambda$ are separable by Theorem~\ref{HMP}, while the separability of $B$ follows from Theorem~\ref{Theorem 3}.  Hence $G$ is also separable as the product of three separable spaces. 
\end{proof}

Since every connected locally compact group is $\sigma$-compact, the last part of Theorem~\ref{Theorem 10} admits the following slightly more general form which will be applied in the proof of Theorem~\ref{Th:Main_B}:

\begin{lemma}\label{Le:PonGen}
Let $N$ be a closed subgroup of an $\omega$-narrow topological group $G$. If the quotient space $G/N$ is locally compact, then the inequality $w(G/N)\leq\cont$ is equivalent to the separability of $G/N$.
\end{lemma}

\begin{proof}
Let $\pi\colon G\to G/N$ be the quotient mapping of $G$ onto the locally compact left coset space $G/N$. If $G/N$ is separable, then Theorem~\ref{Th:1} implies that $w(G/N)\leq\cont$. Assume therefore that $w(G/N)\leq\cont$.

Take an open neighborhood $U$ of the identity element in $G$ such that the closure of the open set $\pi(U)$ in $G/N$ is compact. Since $G$ is $\omega$-narrow, there exists a countable set $C\subset G$ such that $G=CU$. Hence the compact sets $\overline{\pi(xU)}$, with $x\in C$, cover the space $G/N$. This proves that $G/N$ is $\sigma$-compact. 

Let $\{K_n: n\in\omega\}$ be a countable family of compact sets that covers $G/N$. Making use of the local compactness of $G/N$ we can find, for every $n\in\omega$, an open set $O_n$ with compact closure in $G/N$ such that $K_n\subset O_n$. Since the space $G/N$ is normal, there exists a closed $G_\delta$-set $F_n$ in $G/N$ such that $K_n\subset F_n\subset O_n$. It is clear that $F_n$ is compact for each $n\in\omega$. Summing up, each $F_n$ is a compact $G_\delta$-set in the quotient space $G/N$ of the $\omega$-narrow topological group $G$, so Theorem~2 of \cite{Usp1} implies that each $F_n$ is a dyadic compact space. As $w(F_n)\leq w(G/N)\leq\cont$ for each $n\in\omega$, it follows from Proposition~\ref{Pro:Eng} that each $F_n$ is separable. The inclusions $K_n\subset F_n$ with $n\in\omega$ imply that $G/N=\bigcup_{n\in\omega} F_n$, so the space $G/N$ is separable. 
\end{proof}

\begin{remark}\label{Re:Corr}
{\rm Theorem~\ref{Theorem 10}  would not be valid in $ZFC$ if one replaced the condition $w(G)\leq\cont$ by the weaker one, $|G|\leq 2^\cont$. Indeed, the compact topological group $G=\{0,1\}^\kappa$ with $\kappa=\cont^+$ satisfies $w(G)=\kappa$, so $G$ is not separable by Theorem~\ref{Theorem 3}. However, it is consistent with $ZFC$ that $|G|=2^\kappa=2^\cont$ (see  \cite[Chap.~VIII, Sect.~4]{Kun}).}
\end{remark}

Theorem~\ref{Theorem 10} generalizes the second part of Theorem~\ref{Theorem 3}. It is also clear that Theorem~\ref{Theorem 10} is not valid for arbitrary locally compact groups\,\,---\,\,it suffices to take a discrete group of cardinality $\cont$.

A family $\mathcal{N}$ of subsets of a topological space $Y$ is called a \textit{network} for $Y$ if for every point $y \in Y$ and any neighbourhood $U$ of $y$ there exists a set $F \in\mathcal{N}$ such that $y \in F \subset U$. The \textit{network weight} $nw(Y)$ of a space $Y$ is defined as the smallest cardinal number of the form $|\mathcal{N}|$, where $\mathcal{N}$ is a network for $Y$.     

\begin{lemma}\label{Le:Dya} 
If $L$ is a Lindel\"of subspace of a separable Hausdorff space $X$, then $nw(L)\leq\cont$. Hence every compact subspace $K$ of a separable Hausdorff space satisfies $w(K)\leq\cont$.
\end{lemma}

\begin{proof}
Denote by $D$ a countable dense subset of $X$. Let
$$
\mathcal{D}=\bigl\{\overline{C}: C\subseteq D\bigr\}\, \mbox{ and }\, 
\mathcal{N}=\bigl\{\,\bigcap\gamma: \gamma\subseteq\mathcal{D},\ |\gamma|\leq\omega\bigr\}.
$$
Then $|\mathcal{D}|\leq\cont$, $|\mathcal{N}|\leq\cont^\omega=\cont$, and we claim that
the family $\{N\cap L:  N\in\mathcal{N}\}$ is a network for $L$. First we note the family $\mathcal{D}$ separates points of $X$. In other words, for every distinct elements $x,y\in X$, there exists $C\in\mathcal{D}$ such that $x\in C\not\ni y$. This is clear since the space $X$ is Hausdorff. Take a point $x\in L$ and an arbitrary open neighborhood $U$ of $x$ in $X$. Denote by $\mathcal{D}_x$ the family of all $C\in\mathcal{D}$ with $x\in C$. Since $\mathcal{D}$ separates points of $X$, we see that $\bigcap\mathcal{D}_x=\{x\}$. Using the Lindel\"of property of $L$, we can find a countable subfamily $\gamma$ of $\mathcal{D}_x$ such that $L\cap\bigcap\gamma\subseteq L\cap U$. Then $F=\bigcap\gamma\in\mathcal{N}$ and $x\in L\cap F\subseteq L\cap U$. This proves that $\mathcal{N}$ is a network for $L$. 

If $K$ is a compact subset of $X$, then $w(K)=nw(K)$. Since $K$ is obviously Lindel\"of, we conclude that $w(K)\leq\cont$. 
\end{proof}

Combining Lemma~\ref{Le:Dya} and Theorem~\ref{Theorem 3}, we deduce the following fact:

\begin{corollary}\label{Corollary 12} 
If a compact group $G$ is a subspace of a separable Hausdorff space, then $G$ is separable. 
\end{corollary}

It turns out that the compactness of $G$ in Corollary~\ref{Corollary 12} cannot
be weakened to $\sigma$-compactness: 

\begin{example}\label{Example 2} (See \cite[Lemma~3.1]{Comfort_Itzkowitz})
Let $X$ be any separable compact space which contains a closed non-separable subspace $Y$. The free abelian topological group $A(Y)$ naturally embeds into $A(X)$ as a closed subgroup. Then $A(X)$ is a separable $\sigma$-compact group, while $A(Y)$ is not separable\,\,---\,\, otherwise $Y$ would be separable.
\end{example}

The conclusion of Corollary~\ref{Corollary 12} remains valid for connected pro-Lie groups. Later, in Theorem~\ref{Th:Main}, we will show that \lq\lq{connected\rq\rq} can be weakened to \lq\lq{almost connected\rq\rq}.

\begin{proposition}\label{Le:pro-L} 
If a connected pro-Lie group $G$ is a subspace of a separable Hausdorff space $X$, 
then $G$ is separable. 
\end{proposition}

\begin{proof}
By Theorem~\ref{Theorem 7}, the connected pro-Lie group $G$ is homeomorphic to the product $C\times {\R}^\kappa$, where $C$ is a compact connected group and $\kappa$ is a cardinal. Since $C$ can be identified with a subspace of $G$, Lemma~\ref{Le:Dya} implies that $w(C)\leq\cont$. Further, ${\R}^\kappa$ contains a compact subspace homeomorphic with $\{0,1\}^\kappa$. Since the latter space has weight $\kappa$, we apply Lemma~\ref{Le:Dya} once again to conclude that $\kappa\leq\cont$. Hence $w(G)\leq\cont$. Therefore $G$ is separable, by Theorem~\ref{Theorem 10}.
\end{proof}

Since the class of connected pro-Lie groups is productive and contains connected locally compact groups \cite{PROBOOK}, the next fact is immediate from Proposition~\ref{Le:pro-L}.

\begin{corollary}\label{Corollary 13} 
Let $G$ be a product of connected locally compact groups. If $G$ is a subspace of a separable Hausdorff space $X$, then $G$ is separable. 
\end{corollary}

The next result, one of the main in this section, extends Corollary~\ref{Corollary 12} and Proposition~\ref{Le:pro-L} to almost connected pro-Lie groups.

\begin{theorem}\label{Th:Main}
Let $G$ be an almost connected pro-Lie group. If $G$ is homeomorphic to a subspace of a separable Hausdorff space, then it is separable as well.
\end{theorem}

\begin{proof}
By Theorem~\ref{Th_HM}, the group $G$ is homeomorphic to $\R^\kappa \times \{0,1\}^\lambda \times B$, where $\{0,1\}$ is the discrete two-element group, $B$ is a compact connected group, and $\kappa,\lambda$ are cardinals. Assume that $G$ is homeomorphic to a subspace of a separable Hausdorff space. The connected pro-Lie group $\R^\kappa$ is separable by Proposition~\ref{Le:pro-L}. The compact group $K=\{0,1\}^\lambda\times B$ is separable according to Corollary~\ref{Corollary 12}. Therefore $G$ is separable as the product of two separable spaces, $\R^\kappa$ and $K$. (Let us note that $\kappa\cdot\lambda\leq\cont$.)
\end{proof}

In the following result we establish a simple relationship between almost connected pro-Lie groups and $\omega$-narrow groups.

\begin{proposition}\label{Le:Re}
Every almost connected pro-Lie group has countable cellularity and hence is $\omega$-narrow.
\end{proposition}

\begin{proof}
By Theorem~\ref{Th_HM}, every almost connected pro-Lie group $G$ is homeomorphic to the product $\R^\kappa\times\{0,1\}^\lambda\times B$, where $B$ is a compact connected group and $\kappa,\lambda$ are cardinals. The space $\R^\kappa\times\{0,1\}^\lambda$ has countable cellularity as a product of separable spaces \cite[Theorem~2.3.17]{Eng}. The compact group $B$ also has countable cellularity by \cite[Corollary~4.1.8]{AT}. Hence the cellularity of the space $\R^\kappa\times\{0,1\}^\lambda\times B$ is countable (see Corollary~5.4.9 of \cite{AT}). Finally, every topological group of countable cellularity is $\omega$-narrow according to \cite[Theorem~3.4.7]{AT}.
\end{proof}

It turns out that Theorem~\ref{Th:Main} extends to a wider class of topological groups which contains both almost connected pro-Lie groups and locally compact $\sigma$-compact groups. First we need a simple auxiliary result which complements Pontryagin's open homomorphism theorem (see \cite[Theorem~3]{Mor}).

\begin{lemma}\label{Le:aux}
Let $f\colon X\to Y$ be a continuous bijection. If the spaces $X$ and $Y$ are locally compact and $\sigma$-compact and $X$ is homogeneous, then $f$ is a homeomorphism.  
\end{lemma}

\begin{proof}
Let $U_0\subset X$ be a non-empty open set with compact closure. Take a non-empty open set $U$ in $X$ such that $\overline{U}\subset U_0$. Denote by $Homeo(X)$ the family of all homeomorphisms of $X$ onto itself. Since $X$ is homogeneous and $\sigma$-compact, there exists a countable subfamily $\mathcal{A}\subset Homeo(X)$ such that $X=\bigcup\{\alpha(U): \alpha\in\mathcal{A}\}$. Note that the closure of $\alpha(U)$ is compact, for each $\alpha\in\mathcal{A}$.

The family $\{f(\alpha(U)): \alpha\in\mathcal{A}\}$ is a countable cover of $Y$. Since $Y$ is locally compact it has the Baire property. Hence there exists $\alpha\in\mathcal{A}$ such that the closure of $f(\alpha(U))$ has a non-empty interior in $Y$. Let $V$ be a non-empty open set in $Y$ contained in $\overline{f(\alpha(U))}=f(\alpha(\overline{U}))$. Since $f$ is one-to-one, we see that $f^{-1}(V)\subset \alpha(\overline{U})\subset \alpha(U_0)$, so $W=f^{-1}(V)$ is an open subset of $\alpha(U_0)$. The closure of $\alpha(U_0)$ in $X$ is the compact set $\alpha(\overline{U_0})$, so the restriction of $f$ to  
the open subset $W$ of $\alpha(\overline{U_0})$ is a homeomorphism of $W$ onto its image $V=f(W)$. Therefore, by the homogeneity argument, $f$ is a homeomorphism.
\end{proof}

\begin{theorem}\label{Th:Main_B}
Let $G$ be an $\omega$-narrow topological group which contains a closed subgroup $N$ such that $N$ is an almost connected pro-Lie group and the quotient space $G/N$ is locally compact. If $G$ is homeomorphic to a subspace of a separable Hausdorff space, then it is separable as well.
\end{theorem}

\begin{proof}
Assume that $G$ is a subspace of a separable Hausdorff space $X$. Since $N\subset G\subset X$, it follows from Theorem~\ref{Th:Main} that the group $N$ is separable and, hence, $w(N)\leq\cont$ by Theorem~\ref{Th:1}.

Let $\tau$ be the topology of $X$. The family of all regular open sets in $X$ constitutes a base for a weaker topology  on $X$, say, $\sigma$. Since the topology $\tau$ is separable, the space $Y=(X,\sigma)$ has a base of the cardinality at most $\cont$. Indeed, let $S$ be a countable dense subset of $X$.
Then the family 
$$
\mathcal{B}=\{\Int_X\overline{D}: D\subseteq S\}\setminus\{\emptyset\}
$$
is a base for $Y$ and, clearly, $|\mathcal{B}|\leq\cont$. It is also clear that the space $Y$ is Hausdorff. We see in particular that the pseudocharacter of $Y$ is at most $\cont$, i.e. $\psi(Y)\leq\cont$. Since the identity mapping of $X$ onto $Y$ is continuous, it follows that $\psi(X)\leq\psi(Y)\leq\cont$. Hence the subspace $G$ of $X$ satisfies $\psi(G)\leq\cont$ as well. This is the first important property of the group $G$.

We claim that there exists a continuous isomorphism (not necessarily a homeomorphism) $\pi$ of $G$ onto a Hausdorff topological group $H$ with the following properties:
\begin{enumerate}
\item[(i)]   $w(H)\leq\cont$;
\item[(ii)]  the restriction of $\pi$ to $N$ is a topological isomorphism of $N$ onto the closed subgroup     $K=\pi(N)$ of $H$;
\item[(iii)] the quotient space $H/K$ is locally compact.
\end{enumerate}
Indeed, it follows from \cite[Corollary~3.4.19]{AT} that for every neighborhood $U$ of the identity $e$ in $G$, there exists a continuous homomorphism $\pi_U$ of $G$ onto a second-countable Hausdorff topological group $H_U$ such that $\pi_U^{-1}(V)\subseteq U$, for some open neighborhood $V$ of the identity in $H_U$. Let $\mathcal{P}$ be a family of open neighborhoods of $e$ in $G$ such that $\{e\}=\bigcap\mathcal{P}$ and $|\mathcal{P}|\leq\cont$ (we use the fact that $\psi(G)\leq\cont$). Let also $\mathcal{Q}$ be a family of open neighborhoods of $e$ in $G$ such that $|\mathcal{Q}|\leq\cont$ and $\{W\cap N: W\in\mathcal{Q}\}$ is a local base for $N$ at $e$ (here we use the inequality $w(N)\leq\cont$). Then the cardinality of the family $\mathcal{R}=\mathcal{P}\cup\mathcal{Q}$ is not greater than $\cont$. For every element $U\in\mathcal{R}$, we take a continuous homomorphism $\pi_U$ of $G$ onto a second countable topological group $H_U$ as above. Further, since the space $G/H$ is locally compact, there exists an open neighborhood $U_0$ of $e$ in $G$ such that the closure of $\varphi_G(U_0)$ in $G/H$ is compact, where $\varphi_G\colon G\to G/H$ is the quotient mapping. Take a continuous homomorphism $p\colon G\to H_0$ to a second countable topological group $H_0$ such that $p^{-1}(V_0)\subset U_0$, for some open neighborhood $V_0$ of $p(e)$ in $H_0$.

Let $\pi$ be the diagonal product of the family $\{\pi_U: U\in\mathcal{R}\}\cup\{p\}$. Then $\pi$ is a continuous homomorphism of $G$ to the product $P=H_0\times\prod_{U\in\mathcal{R}} H_U$ of second countable Hausdorff topological groups. It follows from our choice of $\mathcal{P}$ and the inclusion $\mathcal{P}\subseteq\mathcal{R}$ that $\pi$ is a monomorphism. Since $|\mathcal{R}|\leq\cont$, the group $P$ and its subgroup $H=\pi(G)$ have weight at most $\cont$. Denote by $p_0$ the projection of $P$ onto the factor $H_0$ and let $W_0=H\cap p_0^{-1}(V_0)$. Then $W_0$ is an open neighborhood of the identity in $H$ and since $p_0\circ\pi=p$, we see that $\pi^{-1}(W_0)=p^{-1}(V_0)\subset U_0$. 

Let us verify that $\pi(N)$ is closed in $H$. It follows from our choice of the family $\mathcal{Q}$ and the inclusion $\mathcal{Q}\subseteq\mathcal{R}$ that the restriction of $\pi$ to $N$ is a topological isomorphism of $N$ onto its image $\pi(N)$. Since the pro-Lie group $N$ is complete, so is the subgroup $K=\pi(N)$ of $H$. Hence $K$ is closed in $H$. In particular, the quotient space $H/K$ is Hausdorff. This proves our claim.

Let $\varphi_G\colon G\to G/N$ and $\varphi_H\colon H\to H/K$ be the canonical quotient mappings onto the left coset spaces $G/N$ and $H/K$, respectively. We define a mapping $i\colon G/N\to H/K$ by letting $i(\varphi_G(x))=\varphi_H(\pi(x))$, for each $x\in G$. Since $K=\pi(N)$, this definition is correct. It follows from our definition of $i$ that the diagram below commutes.
\[
\xymatrix{G\ar@{>}[r]^{\varphi_G\,\,}\ar@{>}[d]_{\pi} &G/N
\ar@{>}[d]^{i}\\
H \ar@{>}[r]^{\varphi_H\,\,} & H/K }
\]
Since $\pi$ is algebraically an isomorphism and $K=\pi(N)$, we see that $i$ is a bijection. The continuity of the mapping $i$ follows from the facts that $\pi$ and $\varphi_H$ are continuous, while $\varphi_G$ is open and continuous. 

The set $\varphi_H(W_0)$ is an open neighborhood of the identity in $H/K$ and the closure of $\varphi_H(W_0)$ is compact, i.e. the space $H/K$ is locally compact. Indeed, it follows from $\pi^{-1}(W_0)\subset U_0$ that $\varphi_H(W_0)\subset i(\varphi_G(U_0))\subset i(\overline{\varphi_G(U_0)})$. Since $\overline{\varphi_G(U_0)}$ is compact, so are the sets $i(\overline{\varphi_G(U_0)})$ and $\overline{\varphi_H(W_0)}$. We have thus proved that the homomorphism $\pi\colon G\to G/H$ satisfies (i)--(iii).

The group $H$ is $\omega$-narrow as a continuous homomorphic image of the $\omega$-narrow group $G$. Hence $H$ can be covered by countably many translations of the open set $W_0$. Since the set $\overline{\varphi_H(W_0)}$ is compact, it follows that the space $H/K$ is $\sigma$-compact.
Similarly, since the group $G$ is $\omega$-narrow and the set $\overline{\varphi_G(U_0)}$ is compact, the space $G/N$ is also $\sigma$-compact. Therefore, both spaces $G/N$ and $H/K$ are locally compact, $\sigma$-compact, and homogeneous.

Finally, it follows from Lemma~\ref{Le:aux} that the bijection $i\colon G/N\to H/K$ is a homeomorphism. It is clear that $w(H/K)\leq w(H)\leq\cont$, so we conclude that $w(G/N)=w(H/K)\leq\cont$. Hence Lemma~\ref{Le:PonGen} implies that the space $G/N$ is separable. Since the subgroup $N$ of $G$ is also separable, the separability of $G$ follows from Remark~\ref{Remark 8}.
\end{proof}

In the sequel we consider embeddings into separable topological groups. As one can expect, the situation improves notably when compared to embeddings into separable Hausdorff spaces.
 
Let us recall that a topological group $G$ is called \textit{feathered} if it contains a nonempty compact subset with a countable neighborhood base in $G$. Equivalently, $G$ is feathered if it contains a compact subgroup $K$ such that the quotient space $G/K$ is metrizable (see \cite[Section~4.3]{AT}). All metrizable groups and all locally compact groups are feathered. Notice that the class of feathered groups is countably productive according to \cite[Proposition~4.3.13]{AT}. 
 
\begin{theorem}\label{Th:Fea}  
Let a feathered topological group $G$ be a subgroup of a separable topological group. Then $G$ is separable. 
\end{theorem}

\begin{proof}
Assume that $G$ is a subgroup of a separable topological group $X$. By \cite[Corollary~3.4.8]{AT}, the group $X$ is $\omega$-narrow. Hence, according to \cite[Theorem~3.4.4]{AT}, the subgroup $G$ of $X$ is also $\omega$-narrow. Applying \cite[4.3.A]{AT}, we deduce that $G$ is Lindel\"of. Take a compact subgroup $K$ of $G$ such that the quotient space $G/K$ is metrizable. Note that the space $G/K$ is Lindel\"of as a continuous image of the Lindel\"of space $G$. Hence $G/K$ is separable.

Finally, the compact  group $K$ is separable by Lemma~\ref{Le:Dya}. Hence the separability of $G$ follows from Remark~\ref{Remark 8}.
\end{proof} 
 
Since all locally compact groups and all metrizable groups are feathered, the following two corollaries are immediate from Theorem~\ref{Th:Fea}; the second of them is well known.  
 
\begin{corollary}\label{Theorem 14}  
If a locally compact topological group $G$ is a subgroup of a separable topological group, then $G$ is separable. 
\end{corollary}

\begin{corollary}\label{Cor:Met}  
If a metrizable group $G$ is a subgroup of a separable topological group, then $G$ is separable. 
\end{corollary}

In fact, the conclusion of Corollary~\ref{Cor:Met} remains valid if $G$ is a subgroup of a topological group $X$ with countable cellularity. Indeed, the group $X$ is $\omega$-narrow by \cite[Theorem~3.4.7]{AT}, and so is its subgroup $G$. Since $G$ is first countable, it follows from \cite[Proposition~3.4.5]{AT} that $G$ has a countable base and hence is separable.

%%%%%%%%%%%%%%%%%%%%%%%%%%%%%%%%%%%%%%%%%%%%%%%%
\begin{remark}\label{Re:Ark}
{\rm 1) A discrete (hence locally compact and metrizable) group $G$ homeomorphic to a closed subspace of a separable Tychonoff space is not necessarily separable. Indeed, it suffices to consider the Niemytzki plane which contains a discrete copy of the real numbers, the $x$-axis. Therefore, Theorem~\ref{Th:Fea} and Corollaries~\ref{Theorem 14} and~\ref{Cor:Met} would not be valid if the group $G$ were assumed to be a subspace of a separable Hausdorff (or even Tychonoff) space. Neither is Corollary~\ref{Theorem 14} valid if $G$ is assumed only to be a pro-Lie group. We will present in Corollary~\ref{Cor:subg} an example of a separable prodiscrete abelian group which contains a closed non-separable subgroup.\smallskip

2) The separable connected pro-Lie group $G=\R^\cont$ contains a closed non-separable subgroup. To see this, we consider the closed subgroup $\Z^\cont$ of $G$. By a theorem of Uspenskij \cite{Usp3}, the group $\Z^\cont$ contains a subgroup $H$ of uncountable cellularity. The closure of $H$ in $G$, say, $K$ is a closed non-separable subgroup of $G$. By Proposition~\ref{Le:pro-L}, the group $K$ cannot be connected.}
\end{remark}

%%%%%%%%%%%%%%%%%%%%%%%%%%%%%%%%%%%%%%%%%%%%%%

A natural question, after Proposition~\ref{Le:pro-L} and Corollary~\ref{Cor:Met}, is whether a connected metrizable group must be separable
if it is a subspace of a separable Hausdorff (or regular) space. Again the answer is \lq{No\rq}. Indeed, consider an arbitrary connected metrizable group $G$ of weight $\cont$. For example, one can take $G=C(X)$, the Banach space of continuous real-valued functions on a compact space $X$ satisfying $w(X)=\cont$, endowed with the sup-norm topology. Since $w(G)=\cont$, the space $G$ is homeomorphic to a subspace of the Tychonoff cube $I^\cont$, where $I=[0,1]$ is the closed unit interval. Thus $G$ embeds as a subspace in a separable regular space, but both the density and weight of $G$ are equal to $\cont$.

%%%%%%%%%%%%%%%%%%%%%%%%%%%%%%%%%%%%%%%%%%%%%%%%%%%%%%%%%%%
\section{Embedding theorems}\label{Sec2}
The following result shows that there is a wealth of separable pseudocompact
topological abelian groups with closed non-separable subgroups. Later,
in Theorem~\ref{Th:nA}, we will present a general version of this result, 
without restricting ourselves to the abelian case.

\begin{theorem}\label{Th:ab}
Every precompact abelian group of weight $\leq\cont$ is topologically
isomorphic to a closed subgroup of a separable, connected, pseudocompact 
abelian group $H$ of weight $\leq\cont$.
\end{theorem}

\begin{proof}
Let $G$ be a precompact abelian topological group satisfying $w(G)\leq\cont$. 
Denote by $\varrho{G}$ the Ra\u{\i}kov completion of $G$. Then $\varrho{G}$
is a compact abelian topological group. Since $G$ is dense in $\varrho{G}$,
we see that $\chi(\varrho{G})=\chi(G)$. Further, the character and weight 
of every precompact (more generally, $\omega$-narrow) topological group 
coincide (see \cite[Corollary~5.2.4]{AT}). Therefore, $w(\varrho{G})=w(G)\leq\cont$.

Since $\varrho{G}$ is a compact abelian topological group of weight 
$\leq\cont$, it is topologically isomorphic to a subgroup of $\Pi=\mathbb{T}^\cont$,
where $\mathbb{T}$ is the circle group with its usual topology inherited
from the complex plane. In particular, we can identify $G$ with a subgroup
of $\Pi$. We will construct the required group $H$ as a dense subgroup of 
$\Pi^\cont$.

For every $\alpha<\cont$, let $\pi_\alpha$ be the projection of $\Pi^\cont$
to the $\alpha$th factor $\Pi_{(\alpha)}$. Let 
$$
\Delta=\{x\in \Pi^\cont: \pi_\alpha(x)=\pi_\beta(x) \mbox{ for all } 
\alpha,\beta\in\cont \}
$$
be the diagonal in $\Pi^\cont$. It is clear that $\Delta$ is a closed
subgroup of $\Pi^\cont$. Further, we can identify $G$ with the subgroup
$$
\widetilde{G}=\{x\in\Delta: \pi_0(x)\in G\}
$$
of $\Delta$. To define the subgroup $H$ of $\Pi^\cont$ we need some
preliminary work.

For every element $x\in\Pi^\cont$, let
$$
\supp(x)=\{\alpha\in\cont: \pi_\alpha(x)\neq e_\Pi\},
$$
where $e_\Pi$ is the identity element of $\Pi$. Then 
$$
\Sigma=\{x\in\Pi^\cont: |\supp(x)|\leq\omega\}
$$
is a dense countably compact subgroup of $\Pi^\cont$ \cite[Corollary~1.6.34]{AT}. 
Hence $\Sigma$ is pseudocompact. 

It is clear that $\Pi^\cont\cong (\mathbb{T}^\cont)^\cont\cong 
\mathbb{T}^{\cont\times\cont}$. We define an element 
$x_0\in\mathbb{T}^{\cont\times\cont}$ as follows. 
Let $\{t_p: p\in\cont\times\cont\}$ be an independent system of
elements of $\mathbb{T}$ having infinite order. In other words, if 
$p_1,\ldots,p_k$ are pairwise distinct elements of $\cont\times\cont$ 
and $n_1,\ldots,n_k$ are arbitrary integers, then the equality
$t_{p_1}^{n_1}\cdots t_{p_k}^{n_k}=1$ in $\mathbb{T}$ implies that
$n_1=\cdots=n_k=0$. Take an element $x_0\in\mathbb{T}^{\cont\times\cont}$
such that $x_0(p)=t_p$ for each $p=(\alpha,\beta)\in\cont\times\cont$. Then
the cyclic subgroup $\hull{x_0}$ of $\Pi^\cont\cong\mathbb{T}^{\cont\times\cont}$
generated by $x_0$ is dense in $\Pi^\cont$ (see the proof of Corollary~25.15 in
\cite{HR}).

Finally we put $H=\widetilde{G}+\Sigma+\hull{x_0}$. Then $H$ is dense 
in $\Pi^\cont$ and separable since $H$ contains the countable dense 
subgroup $\hull{x_0}$ of $\Pi^\cont$. It is also clear that $H$ is
pseudocompact since it contains the dense pseudocompact subgroup 
$\Sigma$. By \cite[Theorem~2.4.15]{AT}, the subgroup $H$ of $\Pi^\cont
\cong\mathbb{T}^{\cont\times\cont}$ is connected. It remains to verify that 
\begin{equation}\label{Eq:1}
H\cap \Delta=\widetilde{G}.
\end{equation}
Indeed, since $\Delta$ is closed in $\Pi^\cont$, it follows from (\ref{Eq:1})
that $\widetilde{G}\cong G$ is closed in $H$. To verify (\ref{Eq:1}), we 
need the following fact:\medskip

\noindent
{\bf Claim.} \textit{If $\alpha<\beta<\cont$ and $n\neq 0$ is an integer,
then $\pi_\alpha(x_0^n)\neq \pi_\beta(x_0^n)$}.\medskip 

We shall now use proof by contradiction. Suppose that $\pi_\alpha(x_0^n)=\pi_\beta(x_0^n)$ 
for distinct $\alpha,\beta\in\cont$ and an integer $n\neq 0$. Let
$J=\{\alpha,\beta\}\subset \cont$ and denote by $\pi_J$ the projection
of $\Pi^\cont$ onto $\Pi^J$. Since $\hull{x_0}$ is dense in $\Pi^\cont$,
the cyclic subgroup $\hull{y_0}$ of $\Pi^J$ is dense $\Pi^J$, where 
$y_0=\pi_J(x_0)$. It follows from $\pi_\alpha(x_0^n)=\pi_\beta(x_0^n)$
that the element $z_0=(\pi_\alpha(x_0))^{-1}\cdot\pi_\beta(x_0)$ of 
$\Pi$ satisfies $z_0^n=e_\Pi$. Denote by $p$ the canonical homomorphism
of $\Pi$ onto $\Pi/\hull{z_0}$. Let $p^2$ be the canonical homomorphism
of $\Pi^J$ onto $\Pi/\hull{z_0}\times\Pi/\hull{z_0}$. Since $\pi_\beta(x_0)=
z_0\cdot\pi_\alpha(x_0)$, we see that $p(\pi_\alpha(x_0))=p(\pi_\beta(x_0))$.
Hence the coordinates of the element $p^2(y_0)=(p(\pi_\alpha(x_0)),p(\pi_\beta(x_0)))$ 
coincide and the diagonal of $\Pi/\hull{z_0}\times\Pi/\hull{z_0}$ contains
the group $p^2(\hull{y_0})$. Since $\hull{y_0}$ is dense in $\Pi^J$ and
the homomorphism $p^2$ is continuous, it follows that the diagonal in 
$\Pi/\hull{z_0}\times\Pi/\hull{z_0}$ is dense in $\Pi/\hull{z_0}\times\Pi/\hull{z_0}$.
This contraction finishes the proof of the Claim. 

Let us turn back to the proof of the equality (\ref{Eq:1}). Take $g\in\widetilde{G}$, 
$s\in\Sigma$, and $n\in\mathbb{Z}$ such that $d=g\cdot s\cdot x_0^n\in\Delta$. 
Since $g\in\Delta$, it follows that $d\cdot g^{-1}=s\cdot x_0^n\in\Delta$.
The element $s\in\Sigma$ has al most countably many coordinates distinct
from $e_\Pi$, so the last equality implies that all the coordinates of
the element $x_0^n$, except for countably many of them, coincide. 
By the above Claim, this is possible only if $n=0$. Since 
$s=s\cdot x_0^n\in\Delta$, we see that $s\in\Sigma\cap\Delta$, so
$s$ is the identity element of $\Pi^\cont$ and $d=g\in\widetilde{G}$.
This proves (\ref{Eq:1}) and completes our argument.
\end{proof}

Theorem~\ref{Th:ab} remains valid in the non-abelian case, but our 
argument becomes considerably more complicated. First we need an 
auxiliary lemma.

\begin{lemma}\label{Le:tec2}
There exists a sequence $\{\varphi_m: m\in\omega\}$ of mappings of 
$\omega$ to $\omega$ satisfying the following condition: For every
integer $k\geq 1$ and every vector triple $(\overline{m},\overline{\jmath}_1,
\overline{\jmath}_2)$, where $\overline{m}=(m_1,\ldots,m_k)$, $\overline{\jmath}_1
=(j_{1,1},\ldots,j_{k,1})$, and $\overline{\jmath}_2=(j_{1,2},\ldots,j_{k,2})$ 
are elements of $\omega^k$ and $m_1,\ldots,m_k$ are pairwise 
distinct, there exists $n\in\omega$ such that $\varphi_{m_i}(n)=j_{i,1}$ 
and $\varphi_{m_i}(n+1)=j_{i,2}$ for each $i$ with $1\leq i\leq k$.
\end{lemma}

\begin{proof}
Let us enumerate the family of all triples $(\overline{m},
\overline{\jmath}_1,\overline{\jmath}_2)$ as in the lemma in such a way 
that if a triple has number $n$ and its first entry is $\overline{m}=
(m_1,\ldots,m_k)$, then $m_1\ldots,m_k$ are less than or equal 
to $n$. Assume that at a stage $n$ of our inductive construction 
we have defined the values $\varphi_m(j)$ for all $m<n$ and $j\leq 2n$.
We now consider the triple $(\overline{m},\overline{\jmath}_1,\overline{\jmath}_2)$ 
which has number $n$ in our enumeration and note that the coordinates 
$m_1,\ldots,m_k$ of $\overline{m}$ are less than or equal to $n$. 
According to the conclusion of the lemma, we have to put 
$\varphi_{m_i}(2n+1)=j_{i,1}$ and $\varphi_{m_i}(2n+2)=j_{i,2}$ 
for each $i$ with $1\leq i\leq k$. The rest of the values $\varphi_m(j)$
with $m\leq n$ and $j\leq 2n+2$ can be chosen arbitrarily. 

Continuing this way, we obtain the sequence $\{\varphi_m: m\in\omega\}$
satisfying the condition of the lemma.
\end{proof}

\begin{theorem}\label{Th:nA}
Every precompact topological group of weight $\leq\cont$ is topologically isomorphic to a closed subgroup of a separable, connected, pseudocompact group $H$ of weight $\leq\cont$.
\end{theorem}

\begin{proof}
It is known that every second countable compact topological group is 
topologically isomorphic to a subgroup of the group $\U=\prod_{n\in\N} U(n)$,
where $U(n)$ is the group of unitary $n\times n$ matrices with complex entries 
for each $n\in\N$. Hence every compact topological group is topologically 
isomorphic to a subgroup of some power of the group $\U$. In particular, if 
$G$ is a precompact topological group and $w(G)\leq\cont$, then the Ra\u{\i}kov 
completion of $G$, $\varrho{G}$, is a compact topological group of weight 
$\leq\cont$, so $\varrho{G}$ and $G$ are topologically isomorphic to subgroups 
of the compact group $\Pi=\U^\cont$. As $\U$ is connected, so is $\Pi$.

Let us identify $G$ with a subgroup of $\Pi$. Since $\U$ is second countable,
the group $\Pi$ is separable by the Hewitt--Marczewski--Pondiczery theorem.
Let $S=\{s_n: n\in\omega\}$ be a dense subset of $\Pi$, where each $s_n$ is
distinct from the identity element $e_\Pi$ of $\Pi$. First we are going to
define a special countable dense subset of $\Pi^\cont$ modifying the original
construction of Hewitt--Marczewski--Pondiczery. To this end, we replace the 
index set $\cont$ with $\R$, so we will work with $\Pi^{\R}$ in place of
$\Pi^\cont$. 

Let $\mathcal{B}$ be the base for the usual topology on $\R$ which consists
of the intervals $(a,b)$ with rational endpoints $a,b$. 

We are now in the position to define a special countable dense subset $D$
of $\Pi^\R$. Let 
\begin{align*}
\mathcal{A} = \{ & (U_1,\ldots,U_k, s_{i_1},\ldots,s_{i_k}): k\geq 1,\ U_1,\ldots,U_k 
\mbox{ are pairwise disjoint }\\ 
& \mbox{ elements of } \mathcal{B}, \mbox{ and } i_1,\ldots,i_k\in\omega \}.
\end{align*}
It is clear that $\mathcal{A}$ is countable. Let us enumerate this family as
$\mathcal{A}=\{L_n: n\in\omega\}$. We also choose a pairwise disjoint sequence 
$\{V_m: m\in\omega\}$ of open unbounded subsets of $\R$. Let $\{\varphi_m: 
m\in\omega\}$ be a sequence of mappings of $\omega$ to $\omega$ satisfying 
the conclusion of Lemma~\ref{Le:tec2}. Given $n\in\omega$ and
$L_n=(U_1,\ldots,U_k, s_{i_1},\ldots,s_{i_k})$ in $\mathcal{A}$, we define 
an element $a_n\in\Pi^\R$ as follows: 
$$
a_n(r)=\begin{cases} s_{i_j} &\text{if $r\in U_j$ for some $j\leq k$};\\
s_{\varphi_n(j)} &\text{if $r\in V_j\setminus\bigcup_{i\leq k} U_k$ 
   for some $j\in\omega$};\\
e_\Pi&\text{otherwise},
\end{cases}
$$
where $r\in\R$. A  standard argument shows that the countable set 
$D=\{a_n: n\in\omega\}$ is dense in $\Pi^\cont$.

Let also $\Sigma$ be the $\Sigma$-product of continuum many copies of the 
group $\Pi$ considered as a subgroup of $\Pi^\R$. As in Theorem~\ref{Th:ab},
$\Sigma$ is a dense pseudocompact subgroup of $\Pi^\R$. 

Denote by $\Delta$ the diagonal subgroup of $\Pi^\R$:
$$
\Delta=\{x\in\Pi^\R: x(r)=x(s) \mbox{ for all } 
r,s\in\R\}.
$$
It is clear that 
$$G_0=\{x\in\Delta: x(0)\in G\}
$$
is an isomorphic topological copy of $G$. 

Denote by $E$ the subgroup of $\Pi^\R$ generated by the set $G_0\cup D$.   
Finally we put $H=E\cdot\Sigma$. Then $H$ is a subgroup of $\Pi^\R$ since 
$\Sigma$ is an invariant subgroup of $\Pi^\R$. As in the proof of Theorem~\ref{Th:ab}, 
it is easy to see that $H$ is a separable, dense, pseudocompact subgroup of 
$\Pi^\R\cong \U^{\R\times\cont}$. Since the group $\U$ is compact and metrizable,
the subgroup $H$ of the connected group $\Pi^\R$ is also connected according to \cite[Theorem~2.4.15]{AT}.

The main difficulty here is to verify the equality $H\cap\Delta=G_0$ analogous 
to (1) in the proof of Theorem~\ref{Th:ab}. Once this is done, we will immediately 
conclude that $G_0\cong G$ is closed in $H$.

Assume that $d=w\cdot b\in\Delta$, where $w\in E$ and $b\in\Sigma$.
We have to show that $d\in G_0$. Let $B=\supp(b)$. Then $w(x)=w(y)$ 
for all $x,y\in\R\setminus B$. Further, the element $w\in E$ has the 
form $w=g_1d_1^{\epsilon_1}\cdots g_nd_n^{\epsilon_n}g_{n+1}$, where 
$g_1,\ldots,g_n,g_{n+1}\in G_0$, $d_1,\ldots,d_n\in D$, and $\epsilon_1,
\ldots,\epsilon_n\in\{1,-1\}$. Here some (or even all) of the elements 
$g_1,\ldots,g_n,g_{n+1}$ can be equal to the identity element $e$ of
$\Pi^\R$ and some of $d_i$ can coincide. 

The word $w$ has the form $w=g_1d_1^{\epsilon_1}\cdots g_nd_n^{\epsilon_n}g_{n+1}$,
where $n\geq 1$. Substituting $d_i$ with variables $z_j$ and assigning 
the same variable to $d_{i_1}$ and $d_{i_2}$ whenever $d_{i_1}=d_{i_2}$, we 
obtain the word $w[\overline{z}]=g_1z_1^{\epsilon_1}\cdots g_nz_p^{\epsilon_n}g_{n+1}$, 
where $\overline{z}=(z_1,\ldots,z_k)$, $k\leq n$, and $1\leq p\leq k$. 
For example, $k=1$ if and only if all the $d_i$ are equal, and $k=n$
if all the $d_i$ are pairwise distinct. Let 
$$
R=\{w[\overline{z}]: \overline{z}=(z_1,\ldots,z_k)\in (\Pi^\R)^k\}
$$
be the range of $w[\,\cdot\,]$. We consider two cases.\smallskip

\noindent Case~1. All values of $w(\overline{z})$ coincide, i.e.
$|R|=1$. Let us take $\overline{z}_0=(e,\ldots,e)$, i.e. $z_1=\cdots=z_k=e$. 
Then $w=w[\overline{z}_0]=g_1\cdots g_n\cdot g_{n+1}\in G_0$.
It now follows from $d=w\cdot b\in\Delta$ that $b=w^{-1}d\in\Delta\cap\Sigma$,
and we conclude that $b=e$ and $d=w\in G_0$.\smallskip

\noindent Case~2. $|R|\geq 2$. Then we choose $\overline{z}_1=(z_{1,1},
\ldots,z_{k,1})$ and $\overline{z}_2=(z_{1,2},\ldots,z_{k,2})$ in 
$(\Pi^\R)^k$ such that 
$w[\overline{z}_1]\neq w[\overline{z}_2]$. We claim 
that there exist $x,y\in\R\setminus B$ such that $w(x)\neq w(y)$, which is 
a contradiction. 

Indeed, take $r\in\R$ such that $w[\overline{z}_1](r)\neq w[\overline{z}_2](r)$.
Hence 
\begin{align*}
g_1(r)z_{1,1}(r)^{\epsilon_1}\cdots z_{p,1}(r)^{\epsilon_n}g_{n+1}(r) &\neq\\
g_1(r)z_{1,2}(r)^{\epsilon_1}\cdots z_{p,2}(r)^{\epsilon_n}g_{n+1}(r) &.
\end{align*}

Since multiplication and inversion in $\Pi$ are continuous, we can find 
an open symmetric neighborhood $W$ of $e_\Pi$ in $\Pi$ such that the sets 
$$
O_\delta=g_1(r)\Bigl(z_{1,i}(r)^{\epsilon_1}W\Bigr)\cdots 
g_n(r)\Bigl(z_{n,i}(r)^{\epsilon_n}W\Bigr)g_{n+1}(r)
$$
with $\delta=1,2$ are disjoint. Let us put $t_{i,\delta}=z_{i,\delta}(r)$ 
for all $i=1,\ldots,k$ and $\delta=1,2$. Since the set $S=\{s_j: 
j\in\omega\}$ is dense in $\Pi$, we can choose, for every $i\leq n$ 
and $\delta=1,2$, an element $s_{i,\delta}\in S\cap t_{i,\delta}W\cap 
Wt_{i,\delta}$. Then $s_{i,\delta}^\epsilon\in t_{i,\delta}^{\epsilon}W$ 
for each $\epsilon=\pm1$. Furthermore, according to our choice of the 
variables $z_i$, the elements $s_{i,\delta}$ can be chosen to satisfy 
$s_{i,\delta}=s_{l,\delta}$ whenever $d_i=d_l$, where $i,l\leq n$ 
and $\delta=1,2$.

It now follows from $O_1\cap O_2=\emptyset$ that the elements
\begin{equation}\label{Eq:2}
h_1=g_1(r)s_{1,1}^{\epsilon_1}\cdots g_i(r)s_{i,1}^{\epsilon_i}\cdots
g_n(r)s_{n,1}^{\epsilon_n}g_{n+1}(r)\in O_1
\end{equation} 
and 
\begin{equation}\label{Eq:3}
h_2=g_1(r)s_{1,2}^{\epsilon_1}\cdots g_i(r)s_{i,2}^{\epsilon_i}\cdots 
g_n(r)s_{n,2}^{\epsilon_n}g_{n+1}(r)\in O_2
\end{equation} 
are distinct. Let us recall that each $d_i$ is in 
$D=\{a_m: m\in\omega\}$, so for every $i\leq n$ there exists 
$m\in\omega$ such that $d_i=a_m$. Since the set $\{d_i: 1\leq i\leq n\}$ 
contains exactly $k$ pairwise distinct elements, there are pairwise
distinct non-negative integers $m_1,\ldots,m_k$ such that 
$\{d_i: 1\leq i\leq n\}=\{a_{m_1},\ldots,a_{m_k}\}$ and $a_{m_j}$
corresponds to the variable $z_j$ for $j=1,\ldots,k$. Similarly,
choose integers $j_{i,\delta}$ for $1\leq i\leq k$ and $\delta=1,2$ 
such that $\{s_{i,\delta}: 1\leq i\leq n\}=\{s_{j_{1,\delta}},\ldots,
s_{j_{k,\delta}}\}$ for each $\delta=1,2$, where both $s_{j_{i,1}}$ 
and $s_{j_{i,2}}$ correspond to the variable $z_i$, $1\leq i\leq k$.  

Consider the $k$-tuples $(m_1,\ldots,m_k)$, $(j_{1,1},\ldots,j_{k,1})$, 
and $(j_{1,2},\ldots,j_{k,2})$.
\newline
By Lemma~\ref{Le:tec2}, there exists
$n_0\in\omega$ such that $\varphi_{m_i}(n_0)=j_{i,1}$ and $\varphi_{m_i}(n_0+1)=
j_{i,2}$ for all $i=1,\ldots,k$. Let $L_{n_0}=(U_1,\ldots,U_l,s_{i_1},
\ldots,s_{i_l})$. Take $x\in V_p\setminus (B\cup\bigcup_{i\leq l} U_i)$ 
and $y\in V_{p+1}\setminus (B\cup\bigcup_{i\leq l} U_i)$, where 
$B=\supp(b)$. This choice of $x$ and $y$ is possible since the sets 
$V_p$ and $V_{p+1}$ are unbounded in $\R$. Then our definition of 
the elements $a_m$ implies that $a_{m_i}(x)=s_{j_{i,1}}$ and 
$a_{m_i}(y)=s_{j_{i,2}}$ for each $i=1,\ldots,k$. Therefore the elements 
$w(x)=h_1\in O_1$ and $w(y)=h_2\in O_2$ of $\Pi$ are distinct (see
the equalities (\ref{Eq:2}) and (\ref{Eq:3})). This contradiction completes 
the proof of the theorem.
\end{proof}

Our next aim is to present an example of a countably compact separable 
abelian group with a closed non-separable subgroup. Our argument makes use 
of an \textit{$\omega$-hereditarily finally dense} subgroup of $\Z(2)^{\omega_1}$
constructed in \cite{HJ} by A. Hajnal and I. Juh\'asz under the assumption of the 
Continuum Hypothesis. Hence our example here also requires $CH$. 

The following lemma is almost evident.

\begin{lemma}\label{Le:pro}
Let $K$ and $L$ be subgroups of a topological abelian group $G$. If
$K$ is countably compact and $L$ is $\omega$-bounded, then $K+L$
is a countably compact subgroup of $G$.
\end{lemma}

\begin{proof}
It is known that the product of a countably compact space and an $\omega$-bounded
space is countably compact (this follows from \cite[Theorem~3.3]{Vau}). Since $K+L$ 
is a continuous image of the countably compact space $K\times L$, the required conclusion is immediate.
\end{proof}

\begin{proposition}\label{Pr:CoCo}
Under $CH$, there exists a separable countably compact topological 
abelian group $G$ which contains a closed non-separable subgroup.
\end{proposition}

\begin{proof}
Let $P=\Z(2)^{\omega_1}$ and $\Pi=P^{\omega_1}$, where both groups 
carry the usual Tychonoff product topology; so $P$ and $\Pi$ are compact 
topological abelian groups. Since $\Pi=(\Z(2)^{\omega_1})^{\omega_1}
\cong \Z(2)^{\omega_1\times\omega_1}\cong \Z(2)^{\omega_1}$, it follows
from \cite[Theorem~2.2]{HJ} that under $CH$, $\Pi$ contains a dense countably
compact \textit{$\omega$-HFD} subgroup $H$ of the cardinality $\cont$. The 
latter means that for every infinite subset $S$ of $H$, there exists a countable 
set $C\subset\omega_1$ such that $\pi_{\omega_1\setminus C}(S)$ is dense 
in $P^{\omega_1\setminus C}$, where $\pi_J$ denotes the projection of 
$P^{\omega_1}$ onto $P^J$ for each non-empty set $J\subset\omega_1$. 
[More precisely, the fact that $H$ is $\omega$-HFD appears on page~202 in
the proof of Theorem~2.2 in \cite{HJ}.] Further, according to \cite[Theorem~2.2]{HJ}, 
the group $H$ is hereditarily separable, i.e. every subspace of $H$ is separable. In particular, $H$ is separable.

Denote by $\Delta$ the diagonal subgroup of $\Pi$, i.e. let
$$
\Delta=\{x\in\Pi: \pi_\alpha(x)=\pi_\beta(x) \mbox{ for all }
\alpha,\beta\in\omega_1\},
$$
where $\pi_\alpha\colon\Pi\to P_{(\alpha)}$ is the projection of $\Pi$
to the $\alpha$th factor. Then $\Delta$ is a closed subgroup of $\Pi$
topologically isomorphic to $P$\,\,---\,\,the projection $\pi_0$ of $\Delta$ 
to $P_{(0)}$ is a topological isomorphism. 

Let $\Sigma_P$ be the $\Sigma$-product of $\omega_1$ copies of the 
group $\Z(2)$ which is identified with the corresponding subgroup of 
$P=\Z(2)^{\omega_1}$. Then $\Sigma_P$ is a proper, dense, countably
compact subgroup of $P$. In fact, $\Sigma_P$ is \textit{$\omega$-bounded},
i.e. the closure in $\Sigma_P$ of every countable subset of $\Sigma_P$ is 
compact \cite[Corollary~1.6.34]{AT}. Further, since $\Sigma_P$ is not 
compact, it cannot be separable. 

Let $G=H+\Sigma_0$, where 
$$
\Sigma_0=\{x\in\Delta: \pi_0(x)\in\Sigma_P\}
$$
is the \lq{diagonal\rq} copy of the group $\Sigma_P$. Again, the groups
$\Sigma_0$ and $\Sigma_P$ are topologically isomorphic. Let us note that
by Lemma~\ref{Le:pro}, $G$ is a countably compact subgroup of $\Pi$. It
is clear that $G$ is separable since it contains a dense separable subgroup 
$H$. 

We claim that the intersection $K=G\cap\Delta$ satisfies $|K:\Sigma_0|<\omega$.
It is clear that $\Sigma_0\subset K$, so it suffices to verify that 
$|H\cap\Delta|<\omega$. Indeed, since the projection of $\Delta$ to an 
arbitrary sub-product $P^J$, with $|J|\geq\omega$, is nowhere dense in 
$P^J$, we see that $H\cap \Delta$ is not finally dense in $\Pi$. Hence
$H\cap\Delta$ is finite. We have thus proved that $|K:\Sigma_0|<\omega$.

It is clear that $K$ is a closed subgroup of $G$. It remains to show 
that the group $K$ is not separable. Since $|K:\Sigma_0|<\omega$ and 
$|\Delta:\Sigma_0|>\omega$, we conclude that $K$ is a proper dense 
subgroup of $\Delta$. Further, there exists a finite subset $F$ of
$K$ such that $K=\Sigma_0+F$. Since $\Sigma_0$ is $\omega$-bounded,
so is $K$. Thus, if $K$ were separable it would be compact, contradicting 
the fact that $K$ is a proper dense subgroup of $\Delta$. 
\end{proof}

Since we have used $CH$ in Proposition~\ref{Pr:CoCo}, it is natural to ask whether a similar construction is possible in $ZFC$ alone:

\begin{question}\label{Example 3} 
Does there exist in $ZFC$ a countably compact separable group $X$ which contains a non-separable closed subgroup?
\end{question}

\section{Prodiscrete groups}\label{Sec:Pd}
Let us recall that a topological group which has a local base at the
identity element consisting of open subgroups is called \textit{protodiscrete.}
A complete protodiscrete group is said to be \textit{prodiscrete.}
Protodiscrete topological groups are exactly the totally disconnected 
pro-Lie groups \cite[Proposition~3.30]{PROBOOK}.

We show in Corollary~\ref{Cor:subg} below that closed subgroups of separable 
prodiscrete abelian groups can fail to be separable. First we prove a general result
on embeddings into separable topological groups. 

\begin{proposition}\label{Pro:12}
A (protodiscrete abelian) topological group $H$ is topologically isomorphic 
to a subgroup of a separable (prodiscrete abelian) topological group if and 
only if $H$ is $\omega$-narrow and satisfies $w(H)\leq\cont$. 
\end{proposition}

\begin{proof}
Assume that a topological group $H$ is a subgroup of a separable topological group $G$. Then $G$ is $\omega$-narrow \cite[Corollary~3.4.8]{AT} and satisfies $w(G)\leq\cont$. Since subgroups of $\omega$-narrow topological groups are $\omega$-narrow \cite[Theorem~3.4.4]{AT}, we see that $H$ is also $\omega$-narrow and satisfies $w(H)\leq\cont$.

Conversely, assume that an $\omega$-narrow group $H$ satisfies $w(H)\leq\cont$. It follows from Theorem~\ref{Th:Gur} (see also \cite[Theorem~3.4.23]{AT}) that $H$ is topologically isomorphic to a subgroup of a topological product $\Pi=\prod_{i\in I} G_i$, where the index set $I$ has the cardinality at most $\cont$ and each factor $G_i$ is a second countable topological group. The group $\Pi$ is separable by the Hewitt--Marczewski--Pondiczery theorem. 

Further, if the group $H$ is protodiscrete, then all factors $G_i$ can be chosen countable and discrete. Indeed, let $\mathcal{N}(e)$ be a local base at the identity element $e$ of $H$ consisting of open subgroups and satisfying $|\mathcal{N}(e)|\leq\cont$. For every $N\in\mathcal{N}(e)$, denote by $\pi_N$ the canonical homomorphism of $H$ onto the discrete quotient group $H/N$. Then the diagonal product of the family $\{\pi_N: N\in\mathcal{N}(e)\}$ is a topological isomorphism of $H$ onto a subgroup of the product group $P=\prod_{N\in\mathcal{N}(e)} H/N$. Since the group $H$ is $\omega$-narrow, each quotient group $H/N$ is countable. Thus $H$ is a topological subgroup of $P$, a product of countable discrete groups. As $|\mathcal{N}(e)|\leq\cont$, the group $P$ is separable. Evidently, the group $P$ is prodiscrete. This completes our argument. 
\end{proof}

\begin{corollary}\label{Cor:subg}
Closed subgroups of separable prodiscrete abelian groups need not be separable.
\end{corollary}

\begin{proof}
Let $D$ be a discrete space of the cardinality $\aleph_1$. Denote by 
$L=D\cup\{x_0\}$ the space which contains $D$ as a dense open subspace 
and in which the sets of the form $L\setminus C$, where $C$ is an arbitrary 
countable subset of $D$, constitute a local base at $x_0$ in $L$. The space 
$L$ is known as a \textit{one-point Lindel\"ofication} of $D$. 

Denote by $H$ the free abelian topological group over $L$. Let us note that $L$ is a Lindel\"of \textit{$P$-space}, i.e. every $G_\delta$-set in $L$ is open. Since all finite powers of $L$ are Lindel\"of, $H$ is Lindel\"of as well. According to \cite[Proposition~7.4.7]{AT}, $H$ is also a $P$-space. Hence $H$ is protodiscrete by \cite[Lemma~4.4.1]{AT}. Notice that $H$ cannot be separable as a non-discrete Hausdorff $P$-space. Every Lindel\"of $P$-group is complete (see \cite[Proposition~2.3]{Tk04}), so the group $H$ is prodiscrete. 

Clearly, the Lindel\"of topological group $H$ is $\omega$-narrow. It is not difficult to verify that the topological character of $H$ (the minimum cardinality of a local base at the identity of $H$) equals $\aleph_1$\,\,---\,\,this follows, for example, from the proof of \cite[Lemma~5.1]{BGHT}. Hence $w(H)=\chi(H)=\aleph_1\leq\cont$ (see \cite[Lemma~5.1.5]{AT}).

We now apply Proposition~\ref{Pro:12} to conclude that $H$ is topologically isomorphic to a subgroup of a separable prodiscrete abelian group $G$. Since $H$ is complete, it is closed in $G$. Thus $G$ contains a closed non-separable subgroup.
\end{proof}

%%%%%%%%%%%%%%%%%%%%%%%%%%%%%%%%%%%%%%%%%%%%%%%%%%%%%%%%%%%%%%%%%%%%%%%
\section{Embeddings into topological groups vs embeddings into regular
spaces}\label{Sec:Emb}%%%%%%%%%%%%%%%%%%%%%%%%%%%%%%%%%%%%%%%%%%%%%%%%%
It is natural to compare the restrictions on a given topological group
$G$ imposed by the existence of either a topological embedding of $G$ into
a separable regular space or a topological isomorphism of $G$ onto a
subgroup of a separable topological group.

First we note that the first of the two classes of topological groups
is strictly wider than the second one. Indeed, consider an arbitrary 
discrete group $G$ satisfying $\omega<|G|\leq\cont$. Then $G$ embeds as
a \textit{closed subspace} into the separable space $\mathbb{N}^\cont$ 
\cite{Eng}, where $\mathbb{N}$ is the set of non-negative integers endowed 
with the discrete topology. However, $G$ does not admit a topological 
isomorphism onto a subgroup of a separable topological group. Indeed, 
every subgroup of a separable topological group is $\omega$-narrow by 
Proposition~\ref{Pro:12}. Since the discrete group $G$ is uncountable, 
it fails to be $\omega$-narrow.

The above observation makes it natural to restrict our attention to $\omega$-narrow topological groups when considering embeddings into separable topological groups. It turns out that in the class of $\omega$-narrow topological groups, the difference between the two types of embeddings disappears, even if we require an embedding to be closed. 

The next result complements Proposition~\ref{Pro:12}. The advantage of Theorem~\ref{Th:Emb} compared to Theorem~\ref{Th:nA} is that in the former one, we manage to identify a wide class of topological groups with closed subgroups of separable path-connected, locally path-connected topological groups. It is not surprising therefore that the Hartman--Mycielski construction \cite{HM} comes into play. 

\begin{theorem}\label{Th:Emb}
The following are equivalent for an arbitrary $\omega$-narrow topological
group $G$:
\begin{enumerate}
\item[{\rm (a)}] $G$ is homeomorphic to a subspace of a separable regular
                 space;
\item[{\rm (b)}] $G$ is topologically isomorphic to a subgroup of a separable
                 topological group;
\item[{\rm (c)}] $G$ is topologically isomorphic to a closed subgroup of a 
                 separable path-connected, locally path-connected topological group.                 
\end{enumerate}
\end{theorem}

\begin{proof}
Since Hausdorff topological groups are regular, it is clear that (c) implies (b) and (b) implies (a). Hence it suffices to show that (a) implies (c). 

Assume that an $\omega$-narrow topological group $G$ is homeomorphic to a subspace of separable regular space $X$. Then $w(X)\leq\cont$ by Theorem~\ref{Th:1} and hence $w(G)\leq w(X)\leq\cont$. Applying Proposition~\ref{Pro:12}, we find a separable topological group $H$ containing $G$ as a topological subgroup. Clearly $G$ can fail to be closed in $H$, so our next step is to construct another separable topological group containing $G$ as a closed subgroup. To this end we will use the path-connected, locally path-connected group $H^\bullet$ corresponding to $H$ and consisting of \textit{step functions} from the semi-open interval $J=[0,1)$ to the group $H$ (for a description of $H^\bullet$, see Hartman and Mycielski \cite{HM} or \cite[Construction~3.8.1]{AT}). 

The group $H$ is canonically isomorphic to a closed subgroup of $H^\bullet$, the corresponding monomorphism $i\colon H\to H^\bullet$ assigns to each element $h\in H$ the constant function $i(h)=h^\bullet\in H^\bullet$ defined by $h^\bullet(x)=h$ for all $x\in J$. 

Let $e$ be the identity of $H$. Denote by $E$ the set of all step functions $f$ from $J$ to $H$ satisfying the following condition:
\begin{enumerate}
\item[(i)] there exists $b\in [0,1)$ such that $f(x)=e$ for each $x$ with $b\leq x<1$.
\end{enumerate}
It is clear that $E$ is a subgroup of $H^\bullet$. Let $D$ be a countable dense subgroup of $H$. Denote by $E'$ the subgroup of $E$ consisting of all $f\in E$ satisfying the following condition:
\begin{enumerate}
\item[(ii)] there exist rational numbers $0=b_0<b_1<\cdots< b_{m-1}<b_m=1$ such that $f$ is constant on each subinterval $[b_k,b_{k+1})$ and $f(b_k)=g_k\in D$ for $k=0,1,\ldots,m-1$. 
\end{enumerate}
Notice that $E'$ is countable. The argument given in the proof of \cite[Theorem~3.8.8,~item~e)]{AT} shows that $E'$ is dense in $H^\bullet$.

Denote by $G_0$ the subgroup of $H^\bullet$ generated by the set $i(G)\cup E$. Since $i\colon H\to H^\bullet$ is a topological monomorphism, the group $G$ is topologically isomorphic to the subgroup $i(G)$ of $H^\bullet$. It also follows from $E'\subset E\subset G_0$ that the group $G_0$ is separable. Let us verify that $i(G)$ is closed in $G_0$. 

First we note that $i(H)$ is closed in $H^\bullet$ according to \cite[Theorem~3.8.3]{AT}. Hence the required conclusion about $i(G)$ will follow if we show that $G_0\cap i(H)=i(G)$. Assume that $f\in G_0\cap i(H)$. Then $f$ is a constant function on $J$ with a single value $h_0\in H$. We have to show that $h_0\in G$, i.e. $f\in i(G)$. As $f\in G_0$, we can write $f$ in the form 
$$
f=i(g_1)^{m_1}t_1^{n_1}\cdots i(g_k)^{m_k}t_k^{n_k}i(g_{k+1})^{m_{k+1}},
$$ 
where $g_1,\ldots,g_k,g_{k+1}\in G$, $t_1,\ldots,t_k\in E$, and $m_i,n_i\in\mathbb{Z}$. Item (i) of our definition of the group $E$ implies that there exists $b<1$ such that $t_i(b)=e$ for each $i=1,\ldots,k$.
Hence $h_0=f(b)=g_1^{m_1}\cdots g_k^{m_k} g_{k+1}^{m_{k+1}}\in G$. Since $f$ is a constant function, we see that $f\in i(G)$. This implies the inclusion $G_0\cap i(H)\subset i(G)$. The inverse inclusion is evident. Therefore, $G\cong i(G)$ is closed in $G_0$. 

Now we have to check that the group $G_0$ is path-connected and locally path-connected. It is worth mentioning that $G_0$ is a proper dense subgroup of $H^\bullet$, but not \textit{every} dense subgroup of $H^\bullet$ inherits these properties from $H^\bullet$. We start with the path-connectedness of $G_0$.

Since $G_0$ is generated by the set $i(G)\cup E$, it suffices to verify that for every element $f\in i(G)\cup E$, there exist a path in $G_0$ connecting the identity $e^\bullet$ of $H^\bullet$ with $f$. Indeed, every element $f\in G_0$ is a product of finitely many elements $f_1,\ldots,f_n$ of $i(G)\cup E$ and, multiplying the paths connecting $f_1,\ldots,f_n$ with $e^\bullet$, we obtain a path in $G_0$ connecting $e^\bullet$ and $f$. So take an arbitrary element $f\in i(G)\cup E$.\smallskip

Case~1. $f\in i(G)$. Then $f=g^\bullet$ for some $g\in G$. For every $r\in [0,1]$, let $f_r$ be a step function from $J$ to $H$ defined by $f_r(x)=g$ if $x<r$ and $f_r(x)=e$ if $r\leq x<1$. It is clear that $f_r\in H^\bullet$ for each $r\in [0,1]$ and that the mapping $\varphi\colon [0,1]\to H^\bullet$, $\varphi(r)=f_r$, is continuous. Since $\varphi(0)=e^\bullet$, $\varphi(1)=f\in i(G)$, and $f_r\in E\subset G_0$ for each $r\in [0,1)$, this proves that $\varphi$ is a path in $G_0$ connecting $e^\bullet$ and $f$. 
\smallskip

Case~2. $f\in E$. Choose a partition $0=b_0<b_1<\cdots<b_m=1$ of $J$ such that $f$ is constant on each interval $[b_i,b_{i+1})$, where $0\leq i<m$, and $f(b_{m-1})=e$. Let us define a path $\varphi\colon [0,1]\to G_0$ connecting $e^\bullet$ with $f$ as follows. First we put $l_k=b_{k+1}-b_k$ for $k=0,\ldots,m-1$. For every $r\in [0,1]$ and every $x\in J$, let 
$$
f_r(x)=\begin{cases} f(x),&\text{if $b_k\leq x<b_k+r\cdot l_k$ for some $k$ with $0\leq k<m$};\\
\,\,\,\,e,&\text{if $b_k+r\cdot l_k\leq x<b_{k+1}$ for some $k$ with $0\leq k<m$.}
\end{cases}
$$
It is easy to verify that $f_0=e^\bullet$, $f_1=f$, and $f_r(x)=e$ if $b_{m-1}\leq x<1$. Hence $f_r\in E\subset G_0$ for each $r\in [0,1]$. Again, the mapping $\varphi\colon [0,1]\to H^\bullet$ defined by $\varphi(r)=f_r$ for each $r\in [0,1]$ is continuous, so $\varphi$ is a path in $G_0$ connecting $e^\bullet$ and $f$.

Summing up, the group $G_0$ is path-connected.

Finally, we check that $G_0$ is locally path-connected. Every neighborhood of $e^\bullet$ in $H^\bullet$ contains an open neighborhood of the form
$$
O(U,\varepsilon)=\{f\in H^\bullet: \mu(\{x\in J: f(x)\notin U\})<\varepsilon\},
$$
where $U$ is an open neighborhood of the identity $e$ in $H$, $\varepsilon>0$, and $\mu$ is the Lebesgue measure on $J$. Therefore, by the homogeneity of $G_0$, it suffices to verify that the intersections $G_0\cap O(U,\varepsilon)$ are path-connected.

Take an arbitrary element $f\in G_0\cap O(U,\varepsilon)$, where $U$ is an open neighborhood of $e$ in $H$ and $\varepsilon>0$. Then $f=i(g_1)^{m_1}t_1^{n_1}\cdots i(g_k)^{m_k}t_k^{n_k}i(g_{k+1}) ^{m_{k+1}}$, where $g_1,\ldots,g_k,g_{k+1}\in G$, $t_1,\ldots,t_k\in E$, and $m_i,n_i\in\mathbb{Z}$. 
Our aim is to define a path $\Phi\colon [0,1]\to G_0\cap O(U,\varepsilon)$ connecting $e^\bullet$ with $f$. We cannot apply directly the formula from the above Case~1 since otherwise we lose control over the measure of the set $\{x\in J: f_r(x)\notin U\}$ for some $r\in (0,1)$, where $f_r$ is assumed to be $\Phi(r)$. Instead, we adjust the speed of changes of the elements $i(g_1),\ldots, i(g_k),i(g_{k+1})$ and $t_1,\ldots,t_k$ on the appropriately chosen subintervals of $J$.  

First we choose a partition $0=b_0<b_1<\cdots<b_{m-1}<b_m=1$ of $J$ such that $t_i$ is constant on $[b_j,b_{j+1})$ for all integers $i\leq k$ and $j<m$. Let also $l_j=b_{j+1}-b_j$, where $j=0,\ldots,m-1$. For every $i=1,\ldots,k,k+1$ we define a path $\varphi_i\colon [0,1]\to H^\bullet$ by
$$
\varphi_i(r,x)=\begin{cases} g_i,&\text{if $b_j\leq x<b_j+r\cdot l_k$ for some $j$ with $0\leq j<m$};\\
\,e,&\text{if $b_j+r\cdot l_j\leq x<b_{j+1}$ for some $j$ with $0\leq j<m$.}
\end{cases}
$$
Then $\varphi_i(0,x)=e$, $\varphi_i(1,x)=g_i$ for each $x\in J$ and $\varphi_i(r,\cdot)\in G_0$ for each $r\in [0,1]$. The path $\varphi_i$ is continuous and connects $e^\bullet$ with $g_i^\bullet$ in $G_0$.

Similarly, we define a path $\psi_i\colon [0,1]\to H^\bullet$ for each $i=1,\ldots,k$ by
$$
\psi_i(r,x)=\begin{cases} t_i(x),&\text{if $b_j\leq x<b_j+r\cdot l_k$ for some $j$ with $0\leq j<m$};\\
\,\,\,\,e,&\text{if $b_j+r\cdot l_j\leq x<b_{j+1}$ for some $j$ with $0\leq j<m$.}
\end{cases}
$$
It is clear that $\psi_i(0,x)=e$, $\psi_i(1,x)=t_i(x)$ for each $x\in J$ and $\psi_i(r,\cdot)\in G_0$ for each $r\in [0,1]$. The path $\psi_i$ is continuous and connects $e^\bullet$ with $t_i$ in $G_0$.

Finally we define a path $\Phi$ in $G_0$ connecting $e^\bullet$ with $f$ by letting 
$$
\Phi(r,x)=\varphi_1(r,x)^{m_1}\cdot\psi_1(r,x)^{n_1}\cdots\varphi_k(r,x)^{m_k}\cdot\psi_k(r,x)^{n_k}\cdot\varphi_{k+1}(r,x)^{n_{k+1}},
$$
where $r\in [0,1]$ and $x\in J$. The path $\Phi$ is continuous being a product of continuous paths $\varphi_i$ and $\psi_i$. The following Claim describes a basic property of the path $\Phi$:\smallskip

\noindent \textbf{Claim.} \textit{For all $r\in [0,1]$ and $x\in J$, either $\Phi(r,x)=f(x)$ or $\Phi(r,x)=e$.} 
\smallskip

Indeed, let $r\in [0,1]$ and $x\in J$ be arbitrary. Choose an integer $j<m$ such that $b_j\leq x<b_{j+1}$. If $b_j\leq x<b_j+rl_j$, then $\varphi_i(r,x)=g_i$ and $\psi_i(r,x)=t_i(x)$ for all $i$, whence it follows that $\Phi(r,x)=f(x)$. If $b_j+r\cdot l_j\leq x<b_{j+1}$, then $\varphi_i(r,x)=e$ and $\psi_i(r,x)=e$
for all $i$, so $\Phi(r,x)=e$. This proves our Claim. 

Applying Claim we see that 
$$
\{x\in J: \Phi(r,x)\notin U\}\subset \{x\in J: f(x)\notin U\}, 
$$
for every $r\in [0,1]$. Hence $\mu(\{x\in J: \Phi(r,x)\notin U\})<\epsilon$ for each $r\in [0,1]$. In other words, the path $\Phi$ lies in $O(U,\varepsilon)$, so the set $O(U,\varepsilon)$ is path-connected.
Since the sets of the form $G_0\cap O(U,\varepsilon)$ constitute a base for $G_0$ at the identity, this completes the proof of the theorem.  
\end{proof}


\begin{thebibliography}{99}

\bibitem{AT} Alexander V.~Arhangel'skii and Mikhail G.~Tkachenko,
\emph{Topological Groups and Related Structures}, 
\newblock Atlantis Series in Mathematics, Vol.~I, Atlantis Press and 
World Scientific, Paris--Amsterdam, 2008.

\bibitem{Comfort} Wistar W.~Comfort, 
\newblock Topological Groups, Chapter 24 in 
\emph{Handbook of Set-Theoretic Topology}, 
\newblock Kenneth Kunen and Jerry E.~Vaughan, eds.,
North-Holland, Amsterdam, New York, Oxford, 1984.

\bibitem{Comfort_Itzkowitz} Wistar W.~Comfort and Gerald L.~Itzkowitz, 
\newblock Density Character in Topological Groups, 
\newblock \emph{Math. Ann.} \textbf{226} (1977), 223--227.

\bibitem{Eng} Richard~Engelking, 
\newblock On the double circumference of Alexandroff,
\newblock \emph{Bull. Acad. Pol. Sci. S\'er. Math.} \textbf{16} (1968), 
629--634.
 
\bibitem{Eng2} Richard~Engelking, 
\newblock Cartesian products and dyadic spaces,
\newblock \emph{Fund. Math.} \textbf{57}, (1965), 287--304.
 
\bibitem{BGHT} Jorge~Galindo,t Mikhail~Tkachenko, 
Montserrat~Bruguera, and Constancio~Hern\'andez, 
\newblock Extensions of reflexive $P$-groups, 
\newblock \emph{Topol. Appl.} {\bf 163} (2014), 112--127. 
 
\bibitem{Gur} Igor I.~Guran, 
\newblock On topological groups close to being Lindel\"of, 
\newblock \emph{Soviet Math. Dokl.} \textbf{23} (1981), 173--175.
Russian original in: \emph{Dokl. Acad. Nauk SSSR} \textbf{256}
(1981), 1035--1037. 
 
\bibitem{HJ} Andr\'as~Hajnal and Istvan~Juh\'asz, 
\newblock A separable normal topological group need not be Lindel\"of, 
\newblock \emph{Gen. Top. Appl.} \textbf{6} (1976), 199--205.
 
\bibitem{HM} S.~Hartman and J.~Mycielski, 
\newblock On embeddings of topological groups into connected 
topological groups, 
\newblock \emph{Colloq. Math.} \textbf{5} (1958), 167--169. 
 
\bibitem{HR} Edwin~Hewitt and Kenneth A.~Ross, 
\newblock \emph{Abstract Harmonic Analysis}, Volume I, 
\newblock Springer-Verlag, Berlin--G\"ottingen--Heidelberg, 1979.

\bibitem{Hodel} Richard E.~Hodel, 
\newblock Cardinal Functions I, Chapter 1 in \emph{Handbook 
of Set-Theoretic Topology}, 
\newblock Kenneth~Kunen and Jerry E.~Vaughan, eds., 
North-Holland, Amsterdam, New York, Oxford, 1984.

\bibitem{PROBOOK} Karl H.~Hofmann and Sidney A.~Morris, 
\newblock \emph{The Lie Theory of Connected  Pro-Lie Groups}, 
\newblock European Mathematical Society, Zurich, 2007.

\bibitem{HM2} Karl H.~Hofmann and Sidney A.~Morris,
\newblock The structure of almost connected pro-Lie groups, 
\newblock \emph{J. Lie Theory} \textbf{21} (2011), 347--383.

\bibitem{Kun} Kenneth~Kunen, 
\newblock \emph{Set Theory}, North Holland, 1980.

\bibitem{Lohman} Robert H.~Lohman and Wilbur J.~Stiles, 
\newblock On separability in linear topological spaces, 
\newblock \emph{Proc. Amer. Math. Soc.} \textbf{42} (1974), 236--237.

\bibitem{Mor} Sidney A.~Morris, 
\newblock \emph{Pontryagin Duality and the Structure of Locally Compact 
Abelian Groups}, 
\newblock London Math. Soc. Lecture Notes Series, 
\textbf{29} Cambridge Univ. Press, 1977.

\bibitem{Tk04} Mikhail G.~Tkachenko, 
\newblock $\mathbb{R}$-factorizable groups and subgroups of 
Lindel\"of $P$-groups, 
\newblock \emph{Topol. Appl.} \textbf{136} (2004), 135--167. 
 
\bibitem{Usp} Vladimir V.~Uspenskij, 
\newblock On the Souslin number of topological groups and their subgroups, 
\newblock \emph{Abstracts of Leningrad Internat. Topol. Conf.}, 
1982; Nauka, Leningrad 1982, p.~162.

\bibitem{Usp1} Vladimir V.~Uspenskij, 
\newblock Compact quotient spaces of topological groups and 
Haydon spectra, 
\newblock \emph{Math. Notes Acad. Sci. USSR} \textbf{42} (1987), 
827--831 (\textit{in Russian}).

\bibitem{Usp3} Vladimir V.~Uspenskij, 
\newblock On the Suslin number of subgroups of products of countable groups,
\newblock \emph{Acta Universitatis Carolinae - Mathematica et Physica} 
\textbf{36} (1995), 85--87. 

\bibitem{Vau} Jerry E.~Vaughan, 
\newblock \emph{Countably Compact and Sequentially Compact Spaces},
\newblock Handbook of Set-Theoretic Topology, K.~Kunen and J.\,E.~Vaughan, 
eds., North-Holland, 1984.

\bibitem{Vidossich} Giovanni Vidossich, 
\newblock Characterization of separability for $LF$-spaces, 
\emph{Ann.\ Inst.\ Fourier}, Grenoble, \textbf{18} (1968), 87--90.

\end{thebibliography}
\end{document}